\documentclass[a4paper]{article}
\usepackage{amsmath,amssymb}
\usepackage[utf8]{inputenc}
\usepackage[dvips]{epsfig}
\usepackage[english,frenchb]{babel}
\usepackage[all]{xy}

\newcommand{\ZZ}{\mathbb{Z}}

\newcommand{\QQ}{\mathbb{Q}}

\newcommand{\Y}{\mathsf{Y}}
\newcommand{\oV}{\mathbf{V}}
\newcommand{\oW}{\mathbf{W}}
\newcommand{\oP}{\mathbf{P}}
\newcommand{\bx}{\epsfig{file=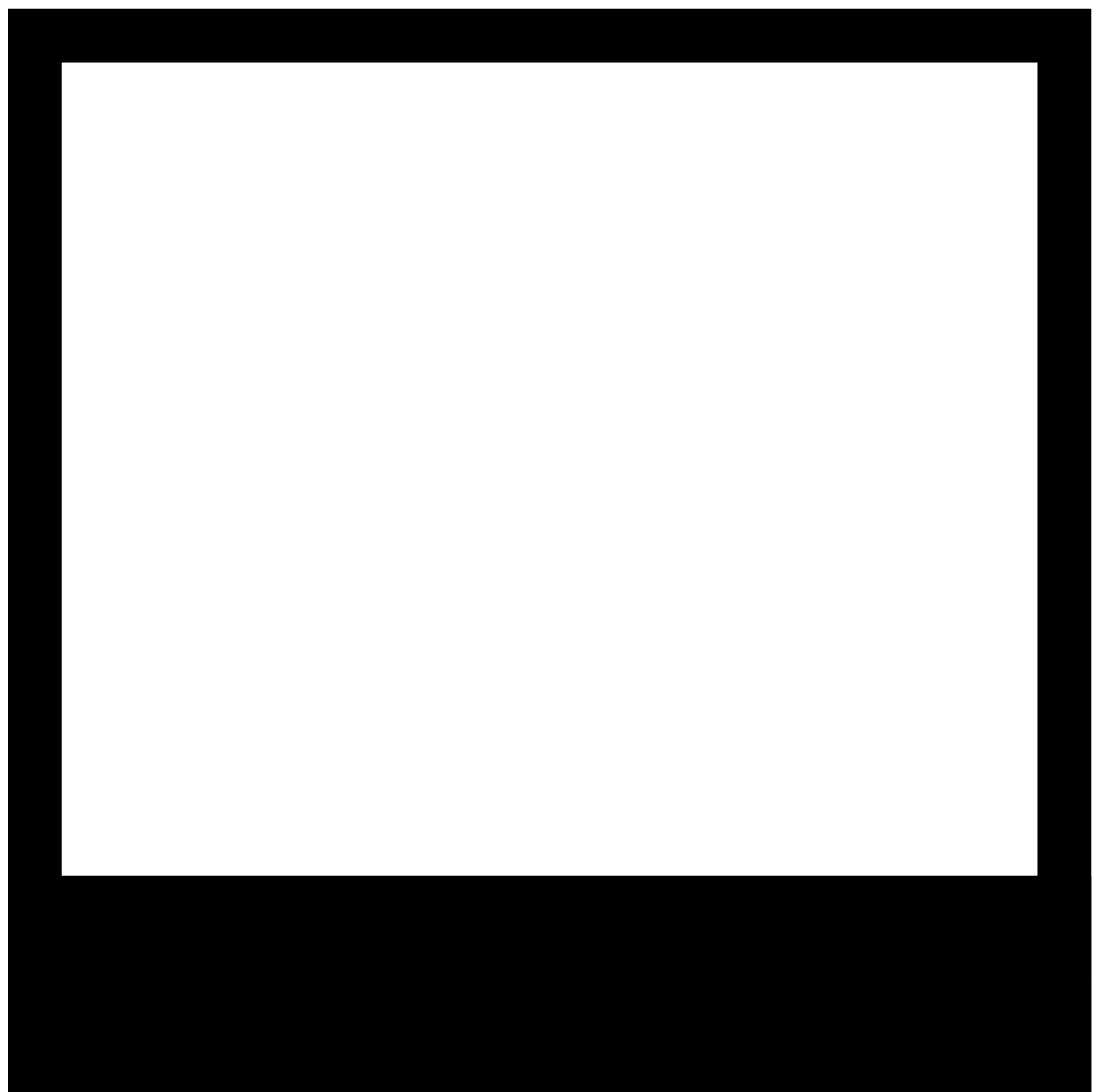,height=2mm}}
\newcommand{\sym}{\mathfrak{S}}

\newcommand{\una}{\epsfig{file=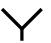,height=2mm}}

\renewcommand{\over}{/}
\renewcommand{\mod}{\operatorname{mod}}
\newcommand{\un}{\boldsymbol{1}}
\newcommand{\under}{\backslash}
\newcommand{\thetav}{\theta_{\oV}}
\newcommand{\id}{\operatorname{Id}}

\newcommand{\Ind}{\operatorname{Ind}}
\renewcommand{\mod}{\operatorname{mod}}
\newcommand{\cha}{\operatorname{Ch}}
\newcommand{\gammav}{\gamma_{\oV}}
\newcommand{\cobar}{\mathsf{B}}

\newtheorem{theorem}{Théorème}[section] 
\newtheorem{proposition}[theorem]{Proposition} 
 
\newtheorem{corollary}[theorem]{Corollaire} 
\newtheorem{lemma}[theorem]{Lemme} 
\newtheorem{definition}{D{\'e}finition}
 
\newtheorem{remark}[theorem]{Remarque}

\newenvironment{proof}{\begin{trivlist}\item{\bf{Preuve.}}}
  {\hfill\rule{2mm}{2mm}\end{trivlist}}

\title{Sur une opérade ternaire \\ liée aux treillis de Tamari}
\author{F. Chapoton}
\date{\today}

\begin{document}

\maketitle

\begin{abstract}
  On introduit une opérade anticyclique $\oV$ définie par une
  présentation ternaire quadratique. On montre qu'elle admet une base
  indexée par les arbres binaires plans. On relie cette construction à
  la famille des treillis de Tamari $(\Y_n)_{n \geq 0}$ en
  construisant un isomorphisme entre $\oV(2n+1)$ et le groupe de
  Grothendieck de la catégorie $\mod \Y_n$ qui envoie la base de
  $\oV(2n+1)$ sur les classes des modules projectifs et qui transforme
  la structure anticyclique de $\oV$ en la transformation de Coxeter
  de la catégorie dérivée de $\mod \Y_n$. La dualité de Koszul des
  opérades permet alors de calculer le polynôme caractéristique de cette
  transformation de Coxeter en utilisant une transformation de
  Legendre.
\end{abstract}

\selectlanguage{english}

\begin{abstract}
  We introduce an anticyclic operad $\oV$ given by a ternary generator
  and a quadratic relation. We show that it admits a natural basis
  indexed by planar binary trees. We then relate this construction to
  the familly of Tamari lattices $(\Y_n)_{n \geq 0}$ by defining an
  isomorphism between $\oV(2n+1)$ and the Grothendieck group of the
  category $\mod \Y_n$. This isomorphism maps the basis of $\oV(2n+1)$
  to the classes of projective modules and sends the anticyclic map of
  the operad $\oV(2n+1)$ to the Coxeter transformation of the derived
  category of $\mod \Y_n$. The Koszul duality theory for operads then
  allows to compute the characteristic polynomial of the Coxeter
  transformation by a Legendre transform.
\end{abstract}

Keywords: operad ; Dendriform operad ; anticyclic operad ; ternary operad ; binary tree ; Tamari lattice ; Coxeter transformation\\

MSC2010: 18D50 ; 05C05 ; 06A11

\selectlanguage{frenchb}

\section{Introduction}

Les arbres binaires plans sont des objets combinatoires très
classiques, qui sont apparus depuis quelques années dans des
situations algébriques variées, dont l'une des plus remarquables est
la description des algèbres dendriformes libres à l'aide d'arbres
binaires plans, due à J.-L. Loday \cite{loday-dialgebras}. Ce résultat
s'exprime, dans la cadre conceptuel des opérades, comme la description
de l'opérade Dendriforme en termes d'arbres binaires plans.

Dans l'étude des algèbres dendriformes libres et de l'opérade
Dendriforme, poursuivie depuis par différents auteurs
\cite{k1,k3,hnt,ronco,aguiar-sottile,arithmetree}, il est
progressivement devenu clair qu'une famille de posets jouait un rôle
fondamental. Ce sont les treillis de Tamari, initialement introduits
par D. Tamari de façon purement combinatoire \cite{tamari} en termes
de parenthèsages. Ces treillis apparaissent en théorie des
représentations de deux façons distinctes, soit comme ordres partiels
sur les modules basculants \cite{riedtmann,happel-unger}, soit parmi les
treillis cambriens, comme ordres partiels sur les amas, dans la
théorie des algèbres amassées de S. Fomin et A. Zelevinsky
\cite{fominz,cambrien}.

Un aspect intriguant de cette relation profonde entre l'opérade
dendriforme et les treillis de Tamari est le point suivant. Si $\Y_n$
est l'ensemble des arbres binaires plans à $n$ sommets, on peut
définir deux applications linéaires de $\ZZ \Y_n$ dans lui-même. La
première application $\tau$ provient de la structure anticyclique de
l'opérade Dendriforme ; elle est en particulier périodique, de période
$n+1$. La seconde application $\theta$ provient de la catégorie $\mod
\Y_n$ des modules sur le poset $\Y_n$. Elle décrit l'action d'un
endofoncteur naturel de la catégorie dérivée $D\mod \Y_n$ de cette
catégorie. Il se trouve que l'application $\tau$ est (au signe près)
le carré de $\theta$ \cite{coxeter_tamari}. En particulier,
l'application $\theta$ est périodique, de période $2n+2$.

Dans un article précédent \cite{categorifi}, cette situation a été
décrite et précisée en introduisant une catégorification de l'opérade
Dendriforme via les catégories de modules sur les treillis de
Tamari. Cette construction utilise de manière essentielle une famille
d'éléments de l'opérade Dendriforme indexée par les arbres
non-croisés. Il se trouve que les arbres non-croisés sont en bijection
avec les arbres ternaires, qui forment une base de l'opérade ternaire
libre sur un générateur. On peut donc se demander si une opérade
ternaire ne pourrait pas jouer un rôle dans le contexte dendriforme.

L'objet du présent article est précisément de présenter une relation
entre les treillis de Tamari et une certaine opérade ternaire. Cette
opérade $\oV$ est engendrée par un élément impair de degré $3$ modulo une
unique relation quadratique. On montre qu'elle admet en degré $2n+1$
une base indexée par les arbres binaires plans à $n$ sommets, en
utilisant la méthode des bases de Gröbner pour les opérades
\cite{hoffbeck,dotsenko}.

On montre que l'opérade $\oV$ possède une structure anticyclique, ce
qui donne une application linéaire $\thetav$ de période $2n+2$ sur le
groupe abélien $\oV(2n+1)$. Le premier résultat principal de l'article est un
isomorphisme entre $\oV(2n+1)$ et $\ZZ \Y_n$ qui identifie $\thetav$
et $\theta$ et donne donc une nouvelle preuve de la périodicité de
$\theta$ mentionnée plus haut.

On utilise cette description de $\theta$ pour en obtenir le
polynôme caractéristique, pour lequel une conjecture a été formulée
dans \cite{chapoton_AIF}. On utilise pour cela une fonction symétrique
associée à chaque opérade cyclique, qui sert à coder l'action du
groupe cyclique sur les composantes de cette opérade. Il se trouve que
les fonctions symétriques associées à deux opérades cycliques duales de Koszul
sont reliées par une transformation de Legendre. Il suffit donc de
calculer la fonction symétrique associée à la duale de Koszul de
$\oV$, qui se trouve être très simple, puis sa transformée de
Legendre. Cette même technique a déjà été employée pour calculer le
polynôme caractéristique de $\tau$ en utilisant sa relation avec
l'opérade Dendriforme.

On dispose donc ainsi de deux opérades dont la structure anticyclique
est fortement reliée aux treillis de Tamari, chacune à sa
manière. L'opérade Dendriforme permet de décrire le carré de $\theta$,
et la base la plus naturelle de cette opérade correspond aux modules
simples sur les treillis de Tamari. L'opérade $\oV$ permet de décrire
$\theta$ et possède une base correspondant aux modules projectifs. Ces
deux structures forment ensemble une structure algébrique assez
complexe et remarquable.

Pour terminer cette introduction, voici quelques mots sur le contexte
en théorie des représentations. La périodicité de la transformation de
Coxeter pour la catégorie $D\mod \Y_n$ se place en fait dans un
ensemble plus vaste de conjectures. Le treillis de Tamari $\Y_n$ peut
en effet être considéré soit comme le treillis cambrien associé aux
carquois de type $A_n$ équiorienté, soit comme le poset des modules
basculants associé aux carquois de type $A_{n+1}$ équiorienté. La
propriété de périodicité semble également vraie pour tous les treillis
cambriens et pour tous les posets de modules basculants associés aux
carquois de type $ADE$.

Pour l'instant, cette propriété est démontrée pour tous les carquois
de type $A$. Elle résulte du cas équiorienté, démontré ici ou dans
\cite{coxeter_tamari} en utilisant la théorie des opérades, et des
théorèmes de Ladkani montrant l'équivalence dérivée entre les treillis
cambriens et les posets de modules basculants \cite{lad1,lad2} lorsque
le carquois change par mutation en un puits ou une source.

Cette périodicité n'est par ailleurs qu'une conséquence d'une
propriété conjecturale plus forte. En effet, elle est vraie si les
catégories dérivées $D \mod \Y_n$ sont Calabi-Yau fractionnaires. Pour
l'instant, on sait seulement que cette propriété plus forte est vraie
pour $n\leq 3$, où les catégories $D \mod \Y_n$ admettent une
description simple en termes de catégories de modules sur les
carquois $A_1,A_2$ et $D_5$.

\section{Construction et propriétés de $\oV$}

\subsection{Généralités et notations}

On se place dans la catégorie monoïdale symétrique des groupes
abéliens gradués, avec la règle des signes de Koszul par rapport à
cette graduation. Les morphismes sont les applications linéaires
respectant la graduation.

Pour le cadre général de la théorie des opérades, on renvoie le
lecteur aux ouvrages \cite{markl,lodayv}.

On ne travaille dans cet article qu'avec des opérades non-symétriques,
qu'on appelle simplement des opérades.

Soit $\oP$ une collection de groupes abéliens gradués
$\oP(n)=\oplus_{k \in\ZZ} \oP(n)_k$ pour $n\geq 1$ et soit
$a\in\oP(n)_k$. On appelle $n$ le degré de $a$, noté $\#a$. On appelle
$k$ le poids de $a$, simplement noté $a$ dans les exposants de $-1$,
par un abus de notation commode qui ne porte pas à confusion. La règle
des signes de Koszul s'applique uniquement à la graduation par le
poids.

On rappelle brièvement la définition des opérades.
\begin{definition}
  Une \textbf{opérade} $\oP$ est la donnée d'une collection de groupes
  abéliens gradués $\oP(n)=\oplus_{k \in\ZZ} \oP(n)_k$ pour $n\geq 1$,
  d'une \textbf{unité} $\un \in \oP(1)_0$ et d'applications linéaires
  \begin{equation}
    \circ_i : \oP(m)\otimes \oP(n) \longrightarrow \oP(m+n-1)
  \end{equation}
  pour $1 \leq i \leq m$. Les \textbf{compositions} $\circ_i$ doivent
  vérifier
  \begin{equation}
    \label{axiome_a}
    (a \circ_i b) \circ_{j+i-1} c = a \circ_i (b \circ_j c),
  \end{equation}
  et 
  \begin{equation}
    \label{axiome_c}
    (a \circ_i b) \circ_{j+\# b-1} c = (-1)^{bc} (a \circ_j c) \circ_i b \quad \text{si }\quad i< j.    
  \end{equation}
\end{definition}

On note $\circ_{\max}$ pour la composition la plus à droite,
\textit{i.e.}
\begin{equation}
  a \circ_{\max} b = a \circ_{\# a} b.
\end{equation}

On rappelle aussi la définition des opérades anticycliques.
\begin{definition}
  Une structure d'\textbf{opérade anticyclique} sur une opérade $\oP$ est la donnée
  pour tout $n \geq 1$ d'un opérateur $\theta$ sur $\oP(n)$ vérifiant
  $\theta(\un)=-\un$, $\theta^{n+1}=\id$ et tels qu'on ait les axiomes suivants :
  \begin{equation}
    \label{cyclique1}
    \theta(a \circ_i b)=\theta(a) \circ_{i-1} b \quad \text{ si }\quad i>1,
  \end{equation}
  et
  \begin{equation}
    \label{cyclique2}
    \theta(a \circ_1 b)=-(-1)^{ab}\theta(b) \circ_{\max} \theta(a).
  \end{equation}
\end{definition}

\begin{remark}
  La notion \textbf{opérade cyclique} est définie similairement, mais
  avec deux changements de signes : $\theta(\un)=\un$ et le coté droit
  de \eqref{cyclique2} est remplacé par son opposé.
\end{remark}

\subsection{L'opérade anticyclique ternaire $\oV$}

Soit $\oV$ l'opérade définie par la présentation par générateurs et
relations suivante.

On se donne un générateur $\bx$ de degré $3$ et de poids $1$ et on
impose la relation
\begin{equation}
  \label{relation_V}
  \bx \circ_1 \bx - \bx \circ_2 \bx + \bx \circ_3 \bx=0.
\end{equation}

On introduit sur $\oV$ une structure anticyclique.

\begin{proposition}
  Il existe sur $\oV$ une unique structure d'opérade anticyclique, donnée
  pour tout $n\geq 0$ par un opérateur $\thetav$ sur $\oV(2n+1)$,
  telle que
  \begin{equation}
    \thetav(\bx)=- \bx.
  \end{equation}
\end{proposition}
\begin{proof}
  Par la théorie générale des opérades anticycliques, il suffit de
  vérifier la compatibilité de $\thetav$ avec la relation
  \eqref{relation_V}, en montrant que l'image par $\thetav$ de cette
  relation est un multiple de cette relation. On obtient, en utilisant
  \eqref{cyclique1} et \eqref{cyclique2} et en gardant les termes dans
  l'ordre initial,
  \begin{equation*}
      \bx \circ_3 \bx + \bx \circ_1 \bx - \bx \circ_2 \bx,
  \end{equation*}
  qui est bien proportionnel à \eqref{relation_V}.
\end{proof}

\subsection{Base $Q$ de $\oV$}

On donne ici une base explicite de $\oV$.

\begin{definition}
  Un \textbf{arbre binaire plan} est un graphe connexe et simplement connexe,
  muni d'un plongement dans le plan considéré à isotopie près, dont
  les sommets sont soit trivalents (\textbf{sommets internes}) soit univalents,
  et muni d'un sommet univalent distingué (\textbf{racine}). On appelle
  \textbf{feuilles} les sommets univalents non distingués.
\end{definition}

On convient de dessiner les arbres binaires plans avec leurs feuilles
en haut et leur racine en bas. Par commodité, on oriente
(implicitement) les arêtes en direction de la racine.

\begin{figure}
  \begin{center}
    \scalebox{0.3}{\includegraphics{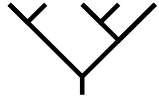}}
    \caption{Un arbre binaire plan dans $\Y_4$}
    \label{arbre}
  \end{center}
\end{figure}

Un arbre binaire plan est soit l'arbre trivial $|$ sans sommet
interne, soit se décompose de manière unique, par enlèvement de
l'arête entrante en sa racine, en une paire $(x,y)$ d'arbres binaires
plans.

Pour $n \geq 0$, on note $\Y_n$ l'ensemble des arbres binaires plans à
$n$ sommets internes.

A chaque élément $x$ de $\Y_n$, on associe un élément $Q_x$ de
$V(2n+1)_{n}$ par récurrence sur $n$. Si $|$ est l'unique arbre sans
sommet, on pose $Q_{|}=\un$, l'unité de l'opérade $\oV$. Sinon, on
associe à l'arbre binaire plan $z$ qui se décompose en une paire $(x,y)$
l'élément
\begin{equation}
  \label{defi_Q}
 Q_z=(-1)^{Q_x+Q_x Q_y}(\bx \circ_3 Q_y) \circ_1 Q_x  = (-1)^{Q_x} (\bx \circ_1 Q_x) \circ_{\max} Q_y.
\end{equation}
L'égalité de ces deux expressions résulte de l'axiome \eqref{axiome_c} des opérades.

Par exemple, on a $Q_{\una}=\bx$ et on associe à l'arbre $x$ de la figure \ref{arbre} l'élément
\begin{equation*}
  Q_x=(\bx \circ_3 (\bx \circ_1 \bx)) \circ_{1} \bx.
\end{equation*}

On note $Q$ l'ensemble des éléments $Q_x$ ainsi définis.

On utilise ci-dessous le formalisme des bases de Gröbner pour les
opérades, voir \cite{hoffbeck,dotsenko} et \cite[\S 8.1]{lodayv}.
\begin{proposition}
  La relation \eqref{relation_V} forme une base de Gröbner de l'idéal
  qu'elle engendre dans l'opérade libre sur $\bx$, pour un ordre
  admissible dans lequel la composition $\circ_2$ domine les
  compositions $\circ_1$ et $\circ_3$.
\end{proposition}
\begin{proof}
  Pour cela, il suffit de montrer une propriété de confluence pour la
  réécriture de la paire critique
  \begin{equation}
    (\bx \circ_2 \bx )\circ_3 \bx = \bx \circ_2 (\bx \circ_2 \bx ).
  \end{equation}
  Il s'agit de vérifier que les deux calculs obtenus en appliquant à volonté
  \begin{equation}
    \label{reecrire}
    \bx \circ_2 \bx \longrightarrow \bx \circ_1 \bx + \bx \circ_3 \bx
  \end{equation}
  à cette expression donnent le même résultat.

  D'une part, la réécriture en partant de $(\bx \circ_2 \bx )\circ_3 \bx $ donne
  \begin{equation*}
    (\bx \circ_1 \bx + \bx \circ_3 \bx) \circ_3 \bx = \bx \circ_1 (\bx \circ_3 \bx) + \bx \circ_3 (\bx \circ_1 \bx).
  \end{equation*}

  D'autre part, la réécriture en partant de $\bx \circ_2 (\bx \circ_2 \bx )$ donne
  \begin{equation*}
    \bx \circ_2 (\bx \circ_1 \bx + \bx \circ_3 \bx) = (\bx \circ_2 \bx) \circ_2 \bx + (\bx \circ_2 \bx) \circ_4 \bx. 
  \end{equation*}
  On utilise à nouveau \eqref{reecrire} deux fois pour obtenir
  \begin{multline*}
    (\bx \circ_1 \bx + \bx \circ_3 \bx) \circ_2 \bx + (\bx \circ_1 \bx + \bx \circ_3 \bx) \circ_4 \bx = \\ 
\bx \circ_1 (\bx \circ_2 \bx) - (\bx \circ_2 \bx) \circ_5 \bx - (\bx \circ_2 \bx) \circ_1 \bx + \bx \circ_3 (\bx \circ_2 \bx).
  \end{multline*}
   On utilise à nouveau \eqref{reecrire} quatre fois pour obtenir
   \begin{multline*}
\bx \circ_1 (\bx \circ_1 \bx + \bx \circ_3 \bx) - (\bx \circ_1 \bx + \bx \circ_3 \bx) \circ_5 \bx \\- (\bx \circ_1 \bx + \bx \circ_3 \bx) \circ_1 \bx + \bx \circ_3 (\bx \circ_1 \bx + \bx \circ_3 \bx).     
   \end{multline*}
   Dans cette somme, certains termes se simplifient par paires et il reste
   \begin{equation*}
     \bx \circ_1 (\bx \circ_3 \bx)+\bx \circ_3 (\bx \circ_1 \bx),     
   \end{equation*}
   ce qui est bien égal à l'autre réécriture.
\end{proof}

\begin{proposition}
  Pour tout $n\geq 0$, l'ensemble $(Q_x)_{x \in \Y_n}$ forme une base
  de $\oV(2n+1)$.
\end{proposition}
\begin{proof}
  La propriété de Gröbner entraîne l'énoncé. On sait en effet qu'une
  base de l'opérade quotient par un idéal est donné par les monômes
  réduits relativement à une base de Gröbner de cet idéal. Dans le cas
  présent, les monômes réduits de l'opérade libre sur $\bx$ par
  rapport à la relation \eqref{relation_V}, pour l'ordre admissible
  choisi, sont exactement les compositions itérées ne faisant pas
  intervenir la composition $\circ_2$, \textit{i.e.} (au signe près)
  les éléments de $Q$.

\end{proof}

\begin{corollary}
  L'opérade $\oV$ est de Koszul.
\end{corollary}
\begin{proof}
  Ceci résulte du fait général qu'une opérade ayant une base de
  Gröbner quadratique est de Koszul, voir par exemple
  \cite[Th. 3.10]{hoffbeck}.
\end{proof}

\subsection{Produits associatifs sur $\oV$}

Par abus de notation, on note aussi $\oV$ la somme directe
\begin{equation}
  \bigoplus_{n \geq 0} \oV(2n+1),
\end{equation}
qui est un groupe abélien doublement gradué par le degré et par le
poids.

On introduit les notations alternatives suivantes :
\begin{equation}
  \label{defo}
  a \over b = (-1)^{ab} b \circ_1 a
\end{equation}
et
\begin{equation}
  \label{defs}
  a * b = a \circ_{\max} b.
\end{equation}

\begin{proposition}
  \label{valise}
  Les produits $\over$ et $*$ sont deux produits associatifs sur $\oV$, qui
  vérifient de plus la relation de compatibilité suivante :
  \begin{equation}
    (a \over b) * c = a \over (b * c).
  \end{equation}
\end{proposition}
\begin{proof}
  L'associativité résulte de l'axiome \eqref{axiome_a} et la
  compatibilité de l'axiome \eqref{axiome_c} des opérades.
\end{proof}

La définition récursive \eqref{defi_Q} de la base $Q$ se traduit, pour
un arbre $z$ qui se décompose en une paire $(x,y)$ d'arbres, en
fonction de $\over$ et $*$ :
\begin{equation}
  \label{redef_Q}
  Q_z = Q_x \over \bx * Q_y,
\end{equation}
où l'on peut omettre les parenthèses par la proposition \ref{valise}.

On introduit deux opérations combinatoires sur les arbres binaires
plans, voir figure \ref{overunder} pour un exemple. Soient $x,y$ deux arbres binaires plans. 

L'arbre binaire plan $x \over y $ est obtenu par greffe à gauche de
$x$ sur $y$, \textit{i.e.} en identifiant l'arête incidente à la
racine de $x$ avec l'arête incidente à la feuille la plus à gauche de
$y$.

L'arbre binaire plan $x \under y $ est obtenu par greffe à droite de $y$
sur $x$, \textit{i.e.} en identifiant l'arête incidente à la racine de
$y$ avec l'arête incidente à la feuille la plus à droite de $x$.

\begin{figure}
  \begin{center}
    \scalebox{0.3}{\includegraphics{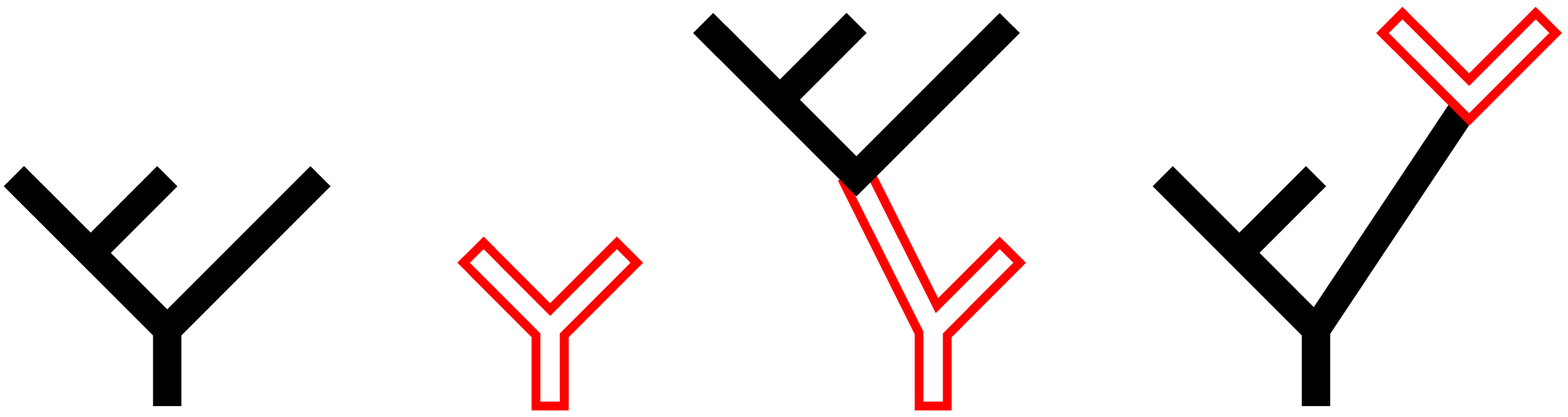}}
    \caption{Les opérations combinatoires $\over$ et $\under$ : $x$, $y$, $x \over y$ et $x \under y$}
    \label{overunder}
  \end{center}
\end{figure}

\begin{remark}
  On a la relation $ (x \over y) \under z = x \over (y \under z)$.
\end{remark}

\begin{lemma}
  \label{descri_prod_Q}
  On a la description suivante des opérations $\over$ et $*$ dans la base $Q$ :
  \begin{align*}
    Q_{x\under y} & = Q_x * Q_y,\\
    Q_{x\over y}& = Q_x \over Q_y.
  \end{align*}
  La famille $Q$ est donc close pour $\over$ et $*$.
\end{lemma}
\begin{proof}
  Ceci résulte facilement de la définition \eqref{redef_Q} de $Q_x$ et de la proposition \ref{valise}, par récurrence.
\end{proof}
\subsection{Caractérisation de $\thetav$}

\begin{proposition}
  \label{carac1}
  L'application $\thetav$ vérifie
  \begin{equation}
   \label{eqcarac1}
     \thetav(a \over b)= - \thetav(a) * \thetav(b).
  \end{equation}
\end{proposition}
\begin{proof}
  On utilise le second axiome des opérades anticycliques \eqref{cyclique2}
  et les définitions \eqref{defo} et \eqref{defs} des produits $\over$
  et $*$.
\end{proof}

\begin{proposition}
  \label{carac2}
  L'application $\thetav$ vérifie
  \begin{equation}
    \label{eqcarac2}
     \thetav(\bx * (a \over (\bx * b))) =  \thetav(\bx * a)\over \thetav(\bx *b) - \thetav(\bx *a) * \thetav(\bx *b).
  \end{equation}
\end{proposition}
\begin{proof}
  En revenant aux expressions en les compositions $\circ_i$, on calcule
  le terme de gauche  en utilisant \eqref{cyclique1}:
  \begin{align*}
    (-1)^{ab+a} \thetav(\bx \circ_3 ((\bx \circ_3 b) \circ_1 a)) &=  -(-1)^{ab+a}\bx \circ_2 ((\bx \circ_3 b) \circ_1 a) \\ &=  -(-1)^{ab+a} ((\bx \circ_2 \bx) \circ_4 b) \circ_2 a.
  \end{align*}
  En utilisant la relation \eqref{relation_V} pour réécrire $\bx \circ_2 \bx$, on obtient
  \begin{equation*}
    -(-1)^{ab+a} ((\bx \circ_1 \bx) \circ_4 b) \circ_2 a  -(-1)^{ab+a}  ((\bx \circ_3 \bx) \circ_4 b) \circ_2 a.
  \end{equation*}
  Par application des axiomes d'opérades, ceci vaut
  \begin{equation*}
    (-1)^{ab+a+b+1} (\bx \circ_2 b) \circ_1 (\bx \circ_2 a) - (\bx \circ_2 a) \circ_{\max} (\bx \circ_2 b).
  \end{equation*}
  En utilisant l'axiome \eqref{cyclique1} des opérades anticycliques, on obtient
  \begin{equation*}
    (-1)^{ab+a+b+1} \thetav(\bx \circ_3 b) \circ_1 \thetav(\bx \circ_3 a) - \thetav(\bx \circ_3 a) \circ_{\max} \thetav(\bx \circ_3 b).
  \end{equation*}
  En repassant aux formules utilisant $\over$ et $*$, on trouve bien
  le second terme voulu.
\end{proof}

\begin{proposition}
  \label{carac_theta_Q}
  Les applications $\thetav$ sont uniquement déterminés par les conditions \eqref{eqcarac1} et \eqref{eqcarac2} et les conditions initiales
  \begin{equation}
    \thetav(\un)=-\un \quad \text{et} \quad    \thetav(\bx)=-\bx.
  \end{equation}
\end{proposition}
\begin{proof}
  Tout d'abord, la valeur de $\thetav(Q_{|})$ pour l'arbre trivial $|$
  est fixé par la première condition initiale, qui fait partie de la
  définition d'une opérade anticyclique.

  On utilise ensuite le fait combinatoire élémentaire suivant. Soit
  $x$ un arbre non trivial. Alors ou bien $x$ peut s'écrire $y \over
  z$ pour deux arbres non triviaux $y$ et $z$, ou bien $x$ peut
  s'écrire $\una \under z$ où $\una$ est l'unique arbre à un seul
  sommet et $z$ est un arbre éventuellement trivial.

  Par conséquent, et par le lemme \ref{descri_prod_Q}, tout élément
  $Q_x$ pour $x$ non trivial peut s'écrire soit comme $Q_y \over Q_z$
  soit comme $\bx * Q_z$ avec les conditions précédemment décrites.

  Si $Q_x$ s'écrit $Q_y \over Q_z$, on peut définir $\thetav$ par
  récurrence par l'équation \eqref{eqcarac1}.

  Sinon $Q_x$ s'écrit $\bx * Q_z$. Si $z$ est trivial, on utilise la
  seconde condition initiale pour définir $\thetav(Q_{\una})$. Sinon,
  on utilise la remarque combinatoire suivante : tout arbre $z$ non
  trivial peut s'écrire
  \begin{equation}
    z = v \over (\una \under w),
  \end{equation}
  où $\una$ est l'unique arbre à un sommet et $v,w$ sont deux arbres
  éventuellement triviaux. Par conséquent, on a
  \begin{equation}
    Q_z = Q_v \over (\bx * Q_w).
  \end{equation}
  On utilise alors \eqref{eqcarac2} pour définir $\thetav(Q_x)$.
\end{proof}

\section{Posets de Tamari}

\subsection{Définition et modules}

Soit $n$ un entier positif ou nul et $\Y_n$ l'ensemble des arbres
binaires plans à $n$ sommets internes. Dans \cite{tamari}, D. Tamari a
défini un ordre partiel sur l'ensemble $\Y_n$ comme suit.

Un arbre $x$ est inférieur ou égal à un arbre $y$ ($x \leq y$) si on
passe de $y$ à $x$ (dans cet ordre) par une suite de mouvements locaux
qui changent une configuration ``arête droite'' en une configuration
``arête gauche'' selon le modèle de la figure \ref{mouvement}.

\begin{figure}
  \begin{center}
    \scalebox{0.15}{\includegraphics{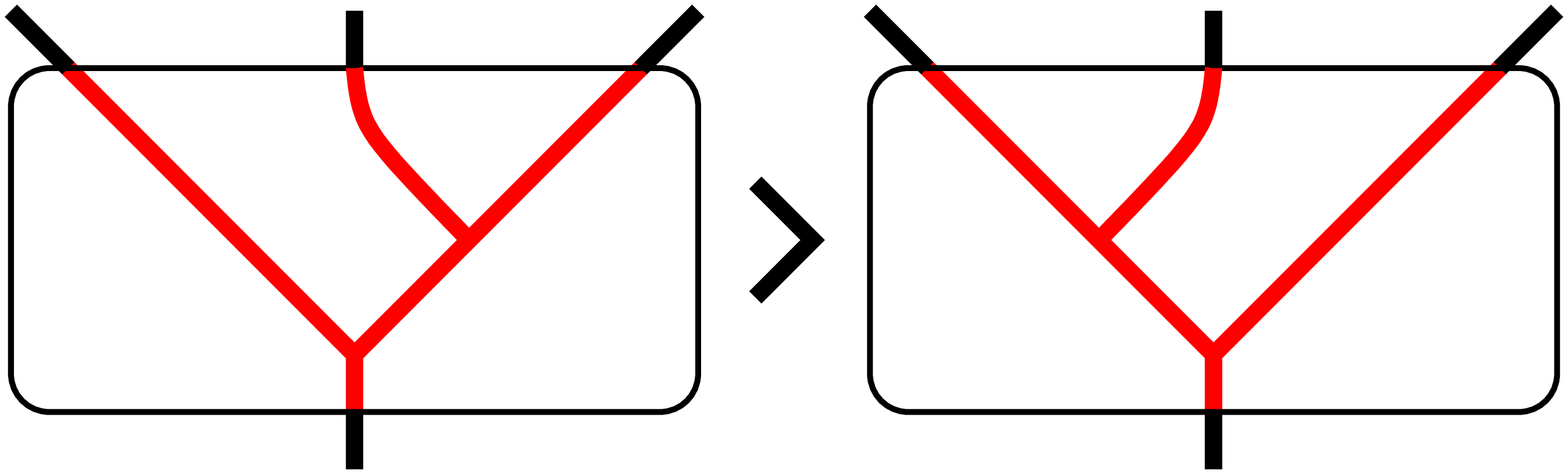}}
    \caption{Changement local : $y > x$}
    \label{mouvement}
  \end{center}
\end{figure}

Friedman et Tamari \cite{friedman-tamari} ont montré que ces posets
sont des treillis. Une autre preuve a été donnée dans \cite{huang-tamari}.

\begin{figure}
  \begin{center}
    \scalebox{0.15}{\includegraphics{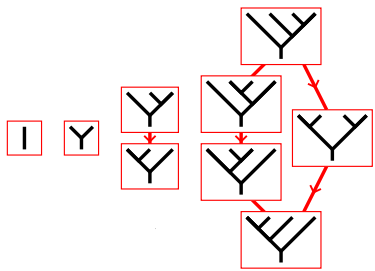}}
    \caption{Treillis de Tamari (minimum en bas)}
    \label{treillis}
  \end{center}
\end{figure}


On renvoie le lecteur aux références
\cite{ladkani_poset,curtis-reiner,happel,lenzing} pour la théorie des
représentations des carquois, des posets et des algèbres de dimension
finie.

On peut considérer le diagramme de Hasse de $\Y_n$ comme un carquois-avec-relations, en orientant les arêtes du maximum vers le minimum, et
en imposant comme relations l'égalité de toute paire de chemins ayant
mêmes débuts et fins.

On fixe désormais un corps de base $\mathsf{k}$ pour les modules sur
les carquois. Les constructions qui suivent en dépendent peu et
la dépendance en $\mathsf{k}$ sera implicite.

On considère alors la catégorie $\mod \Y_n$ des modules sur le
carquois-avec-relations $\Y_n$. Cette catégorie est équivalente à la
catégorie des modules sur l'algèbre d'incidence de $\Y_n$. C'est une
catégorie abélienne de dimension globale finie.

On dispose de trois bases de l'anneau de Grothendieck $K_0(\mod \Y_n)$
données par les classes des modules projectifs, injectifs et
simples sur $\Y_n$, notées respectivement $P_x,I_x$ et $S_x$ pour $x \in\Y_n$.

La base $S$ des simples est reliée à la base $P$ des projectifs par la
formule
\begin{equation}
  P_x = \sum_{y \leq x} S_y,
\end{equation}
et à la base $I$ des injectifs par la formule
\begin{equation}
  I_x = \sum_{y \geq x} S_y.
\end{equation}

\subsection{Structures algébriques}

On rappelle dans cette section des structures algébriques qui sont
déjà bien étudiées, notamment dans les articles
\cite{hnt,loday-ronco,aguiar-sottile,arithmetree}. Pour cette raison,
on ne donne que des indications de preuves.

On considère maintenant la somme directe
\begin{equation}
  \bigoplus_{n \geq 0} K_0(\mod \Y_n)
\end{equation}
et plusieurs produits sur ce groupe abélien.

On a tout d'abord deux produits associatifs $\over$ et $\under$
définis sur la base $S$ par
\begin{align}
  \label{defi_S_o}
  S_x \over S_y &= S_{x\over y},\\
  \label{defi_S_u}
  S_x \under S_y &= S_{x\under y}.
\end{align}

On a par ailleurs un autre produit associatif $*$, qui est le produit
associatif de l'algèbre dendriforme libre. On peut le définir par la
caractérisation suivante, due à Loday et Ronco \cite{loday-ronco}.

\begin{proposition}
  \label{produit_intervalle}
  Soient $x,y$ deux arbres. Alors
  \begin{equation}
    S_x * S_y = \sum_{x\over y \leq z \leq x \under y} S_z.
  \end{equation}
\end{proposition}

On peut par ailleurs donner du produit $*$ une description
combinatoire alternative, voir \cite[Prop. 5.11]{loday-dialgebras}
pour un énoncé qui implique celui ci-dessous.

\begin{proposition}
  \label{shuffle_arbre}
  Le produit $S_x *S_y$ est la somme des $S_z$ où $z$ décrit les
  arbres binaires plans obtenus en identifiant le coté droit de $x$
  avec le coté gauche de $y$ par un homéomorphisme croissant qui
  n'envoie pas de sommet interne sur un sommet interne.
\end{proposition}
 
L'élément $S_{|}$ est une unité pour les trois produits $\over$, $\under$ et $*$.

\begin{proposition}
  \label{descri_prod_P}
  On a la description suivante des opérations $\over$ et $*$ dans la base $P$ :
  \begin{align}
    P_{x\under y}&=P_x * P_y,\\
    P_{x\over y}&=P_x \over P_y.
  \end{align}
\end{proposition}
\begin{proof}
  Pour la première formule, il s'agit de voir que les arbres $z$ qui
  sont inférieurs ou égaux à l'arbre $x\under y$ (condition
  \textbf{A}) sont exactement ceux qui s'obtiennent par recollement du
  bord droit d'un arbre inférieur à $x$ avec le bord gauche d'un arbre
  inférieur à $y$ (condition \textbf{B}). Il est facile de déduire de
  la forme locale de l'ordre de Tamari que la condition \textbf{B} est
  stable par diminution dans le poset de Tamari, donc que la condition
  \textbf{A} entraîne la condition \textbf{B}. Réciproquement, étant
  donné un arbre $z$ vérifiant la condition \textbf{B}, on peut
  l'écrire comme un des termes de $S_{x'} * S_{y'}$ avec $x' \leq x$
  et $y' \leq y$. En appliquant une suite convenable de mouvements
  locaux définissant l'ordre de Tamari, on peut montrer que $z \leq x'
  \under y' $ Comme $ x' \under y' \leq x \under y$, ceci entraîne
  que $z$ vérifie la condition \textbf{A}.

  Montrons la seconde formule, en établissant l'énoncé suivant : pour
  tous $x,y$, on a une bijection
  \begin{align*}
   \{ x' \mid x'\leq x\} \times \{y' \mid y' \leq y\}  &\simeq \{ z \mid z \leq x\over y\}\\
           (x',y') &\mapsto x'\over y'.
  \end{align*}
  Le fait que $x' \over y' \leq x\over y$ résulte de la nature locale
  de l'ordre de Tamari. L'injectivité est claire. Pour montrer la
  surjectivité, il suffit de voir que l'existence pour $z$ d'une
  décomposition de la forme souhaitée entraîne l'existence d'une telle
  décomposition pour tout $z'$ couvert par $z$. On conclut en partant
  de la décomposition de $x \over y$.
\end{proof}

En utilisant l'unique anti-automorphisme des
treillis de Tamari, on obtient la proposition suivante.

\begin{proposition}
  \label{descri_prod_I}
  On a la description suivante des opérations $\under$ et $*$ dans la base $I$:
  \begin{align}
    I_{x\over y}&=I_x * I_y,\\
    I_{x\under y}&=I_x \under I_y.
  \end{align}
\end{proposition}

On note $\theta$ la transformation de Coxeter du poset $\Y_n$. C'est
un endomorphisme de $K_0(\mod \Y_n)$ qui provient d'un endofoncteur
$\tau_{AR}$ de la catégorie dérivée de $\mod \Y_n$.

L'endomorphisme $\theta$ est défini sur $K_0(\mod \Y_n)$ par la formule
\begin{equation}
  \theta(P_x)=-I_x,
\end{equation}
pour tout $x \in \Y_n$.

\begin{proposition}
  \label{carac_theta_P_1}
  Les applications $\theta$ vérifient
  \begin{align}
    \theta(P_{\una})&=-P_{\una},\\
    \theta(a \over b)&=-\theta(a) * \theta(b),\\
    \theta(a * b)&= - \theta(a) \under \theta(b).
  \end{align}
\end{proposition}
\begin{proof}
  La première formule pour $P_{\una}$ est immédiate, car $P_{\una}=I_{\una}$.

  Pour la seconde formule, on calcule
  \begin{equation*}
    \theta(P_x \over P_y)=\theta(P_{x \over y})=-I_{x \over y}=- I_x * I_y=-\theta(P_x) * \theta(P_y).
  \end{equation*}

  Pour la troisième formule, on calcule
  \begin{equation*}
    \theta(P_x * P_y)=\theta(P_{x \under y})=-I_{x \under y}=- I_x \under I_y= - \theta(P_x) \under \theta(P_y).
  \end{equation*}
\end{proof}

\begin{lemma}
  \label{lemme_dendri}
  On a
  \begin{equation}
    (P_{\una} \under a)*(P_{\una} \under b)=P_{\una} \under (a * (P_{\una} \under b))+(P_{\una} \under a)\over (P_{\una} \under b).
  \end{equation}
\end{lemma}
\begin{proof}
  Ceci résulte de la description des trois produits $\over$, $\under $
  et $*$ dans la base $S$, voir \eqref{defi_S_o}, \eqref{defi_S_u} et
  la proposition \ref{shuffle_arbre}. Le produit $(P_{\una} \under
  a)*(P_{\una} \under b)$ se coupe naturellement en deux termes selon
  que le sommet interne inférieur provienne de $P_{\una} \under a$ ou
  de $P_{\una} \under b$.
\end{proof}

\begin{proposition}
  \label{carac_theta_P_2}
  On a 
  \begin{equation}
    \theta(P_{\una} * (a \over (P_{\una} * b))) = \theta(P_{\una} * a)\over \theta(P_{\una} *b) - \theta(P_{\una} *a) * \theta(P_{\una} *b).
  \end{equation}
\end{proposition}
\begin{proof}
  On calcule le terme de gauche :
  \begin{equation*}
    -\theta(P_{\una}) \under \theta(a \over (P_{\una} * b)) = - \theta(P_{\una}) \under (\theta(a) * (\theta(P_{\una}) \under \theta(b))),
  \end{equation*}
  soit
  \begin{equation*}
    - P_{\una} \under (\theta(a) * (P_{\una} \under \theta(b))).
  \end{equation*}
  On calcule le terme de droite :
  \begin{equation*}
    (\theta(P_{\una})\under \theta(a))\over (\theta(P_{\una})\under \theta(b))-(\theta(P_{\una})\under \theta(a)) * (\theta(P_{\una})\under \theta(b)),
  \end{equation*}
  soit
  \begin{equation*}
    (P_{\una}\under \theta(a))\over (P_{\una}\under \theta(b))-(P_{\una}\under \theta(a)) * (P_{\una}\under \theta(b)).
  \end{equation*}
  Par le lemme \ref{lemme_dendri}, on déduit l'égalité voulue.
\end{proof}

\section{Isomorphisme}

Pour tout $n\geq 0$, on définit une application linéaire $\psi$ de
$\oV(2n+1)$ dans $K_0(\mod \Y_n)$ par
\begin{equation}
  \psi(Q_x) = P_x,
\end{equation}
pour tout arbre binaire plan $x$.

\begin{proposition}
  L'application $\psi$ est un isomorphisme de groupes abéliens.
\end{proposition}
\begin{proof}
  En effet, l'ensemble $(P_x)_{x \in \Y_n}$ est une
  base de $K_0(\mod \Y_n) $.
\end{proof}

\begin{proposition}
  L'application $\psi$ est un morphisme pour $\over$ et $*$.
\end{proposition}
\begin{proof}
  Ceci résulte de la description identique de ces produits dans les
  bases $Q$ et $P$ à l'aide des opérations combinatoires $\over$ et
  $\under$ sur les arbres binaires plans, par les propositions
  \ref{descri_prod_Q} et \ref{descri_prod_P}.
\end{proof}


\begin{theorem}
  \label{theo_idem}
  L'isomorphisme $\psi$ de $\oV(2n+1)$ dans $K_0(\mod \Y_n)$ transforme l'endomorphisme $\thetav$ en l'endomorphisme $ \theta$ : 
  \begin{equation}
    \psi \thetav = \theta \psi.
  \end{equation}
\end{theorem}
\begin{proof}
  L'idée de la preuve est la suivante : les morphismes $\thetav$ et
  $\theta$ sont caractérisés par des propriétés similaires qui ne font
  intervenir que les produits $*$ et $\over$.
  
  Pour faciliter la preuve, on identifie $P_x$ avec $Q_x$ pour tout
  $x$. On va donc montrer que $\thetav= \theta$.

  Vérifions d'abord les conditions initiales. Pour $n=0$,
  $\thetav=-\id$ car $\oV$ est une opérade anticyclique et
  $\theta=-\id$. Pour $n=1$, $\thetav=-\id$ et $\theta=-\id$.

  En utilisant les propositions \ref{carac_theta_P_1} et
  \ref{carac_theta_P_2}, on vérifie que la collection de morphismes
  $ \theta$ satisfait, après identification des bases $Q$ et $P$,
  les conditions \eqref{eqcarac1} et \eqref{eqcarac2} qui
  caractérisent $\thetav$ selon la proposition \ref{carac_theta_Q}.
\end{proof}

\section{Application au polynôme caractéristique}

On va utiliser le théorème \ref{theo_idem} et la théorie de la dualité
de Koszul des opérades pour calculer le polynôme caractéristique de la
transformation de Coxeter $\theta$ du poset de Tamari $\Y_n$.

\subsection{Structure cyclique et dualité de Koszul}

Par commodité, on va plutôt utiliser une structure cyclique sur $\oV$,
qui se trouve être donnée par l'opposé de la structure anticyclique.

\begin{proposition}
  On peut munir $\oV$ d'une structure cyclique définie par
  $\gammav(\bx)=\bx$. On a alors la relation globale $\gammav=-\thetav$.
\end{proposition}
\begin{proof}
  La vérification de la compatibilité de $\gammav$ avec la relation
  \eqref{relation_V} est immédiate. En effet, l'image par $\gammav$ de cette
  relation est
  \begin{equation*}
    -\bx \circ_3 \bx - \bx \circ_1 \bx + \bx \circ_2 \bx.
  \end{equation*}
  Pour le reste, il suffit de voir que $-\gammav$ définit une
  structure anticyclique qui coïncide avec $\thetav$ sur $\bx$. Les
  deux axiomes d'opérade cyclique pour $\gammav$ impliquent les deux
  axiomes d'opérade anticyclique pour $-\gammav$. De plus, $-\gammav$
  est bien d'ordre $2n+2$ sur l'espace $\oV(2n+1)$ pour tout $n$.
\end{proof}

Décrivons l'opérade $\oW$ duale de $\oV$ (voir \cite{markl_remm} pour
la procédure de calcul de l'opérade duale). Elle est engendrée par un
élément $w$ de poids $0$ et de degré $3$ vérifiant les relations
suivantes :
\begin{equation}
  w \circ_1 w + w\circ_2 w =0 \quad\text{ et }\quad w \circ_2 w + w\circ_3 w =0. 
\end{equation}

Par dualité de Koszul, comme $\oV$ est de Koszul, $\oW$ l'est
aussi. On a la description suivante :
\begin{proposition}
  L'opérade $\oW$ est de dimension $1$ en chaque degré impair et une
  base de $\oW(2n+1)$ est donnée par $w_n=w \circ_1 \dots \circ_1 w$ avec
  $n$ copies de $w$.
\end{proposition}

\begin{proposition}
\label{cycliqueW}
  On peut munir $\oW$ d'une structure cyclique définie par
  $\gamma(w)=-w$. On a alors $\gamma(w_n)=(-1)^n w_n$.
\end{proposition}
\begin{proof}
  Pour l'existence, il suffit de vérifier la compatibilité de $\gamma$
  avec les relations de l'opérade $\oW$. Pour la valeur de $\gamma$
  sur $w_n$, il suffit de faire une récurrence sur $n$. En effet, on a
  \begin{equation*}
    \gamma(w_n)=\gamma(w_{n-1} \circ_1 w)=- w \circ_3 \gamma(w_{n-1})=(-1)^n w \circ_3 w_{n-1}.
  \end{equation*}
  On montre par ailleurs par récurrence sur $n$ que $ w \circ_3
  w_{n-1}= w_n$, en utilisant la relation $w \circ_1 w = w\circ_3 w$.
\end{proof}

\begin{remark}
  \label{crux}
  On peut vérifier que cette structure cyclique sur $\oW$ est celle
  qui provient, via le quasi-isomorphisme entre $\cobar\oV$ et $\oW$,
  de la structure cyclique naturelle sur la cobar-construction
  $\cobar\oV$.
\end{remark}

\subsection{Rappels sur les fonctions symétriques}

On utilise les notations standards pour les fonctions symétriques, qui
sont celles du livre \cite{macdonald}.

Soit $\Lambda$ l'anneau des fonctions symétriques sur $\QQ$. On note
$(p_n)_{n \geq 1}$ les fonctions symétriques ``sommes de puissances''. L'anneau
$\Lambda$ est l'anneau des polynômes en les $(p_n)_{n \geq 1}$. Il
admet une base $(p_\lambda)_\lambda$, formée des monômes en les $(p_n)_{n \geq 1}$
et indexée par les partitions d'entiers. On munit $\Lambda$ d'une graduation
naturelle en posant $\deg(p_i)=i$. On travaille par la suite dans le complété de l'anneau $\Lambda$ par
rapport à sa graduation. 

On note $\Sigma$ la suspension des fonctions symétriques, définie par
\begin{equation}
  (\Sigma f)(p_1,p_2,\dots,p_i,\dots)=-f(-p_1,-p_2,\dots,-p_i,\dots).
\end{equation}
On note $\omega$ l'automorphisme involutif des fonctions symétriques
qui est défini par
\begin{equation}
  \label{def_omega}
  (\omega f)(p_1,p_2,\dots,p_i,\dots)=f(-p_1,p_2,\dots,(-1)^{i-1} p_i,\dots).
\end{equation}

Si $\rho$ est un caractère du groupe symétrique $\sym_n$, on
identifiera $\rho$ à la fonction symétrique
\begin{equation}
  \sum_{\lambda \vdash n} \rho(C_\lambda) \frac{p_\lambda}{z_\lambda},
\end{equation}
où la somme porte sur les partitions de $n$, $C_\lambda$ est la classe
de conjugaison de type cyclique $\lambda$ et $z_\lambda |C_\lambda| = n!$.

On note $\circ$ le pléthysme des fonctions symétriques, pour lequel on
renvoie à la littérature.

On appelle \textbf{série caractéristique} d'une opérade cyclique $\oP$
la fonction symétrique 
\begin{equation}
  \cha_{\oP} = \sum_{n \geq 1} \Ind_{\ZZ_{/n+1}}^{\sym_{n+1}} \left(\sum_{k \in \ZZ} (-1)^k \chi_{\oP(n)_k}\right),
\end{equation}
où $\chi_{\oP(n)_k}$ est le caractère de $\ZZ_{/n+1}$ sur $\oP(n)_k$,
la composante de poids $k$ de $\oP(n)$. 


\smallskip

On rappelle maintenant la transformation de Legendre des fonctions
symétriques, introduite par Getzler et Kapranov dans \cite[\S
7]{getzler-kapranov_modular}.

Soient $A$ et $B$ deux fonctions symétriques sans termes de degré
inférieur ou égal à $1$ et tels que les termes de degré $1$ de
$\partial_{p_1} A$ et de $\partial_{p_1} B$ sont non nuls. On dit que
$A$ est la transformée de Legendre de $B$ si
\begin{equation}
  A \circ \partial_{p_1} B + B = p_1 \partial_{p_1} B.
\end{equation}
La transformation de Legendre est une involution. Dans cette
situation, on a aussi une relation entre les dérivées partielles de
$A$ et de $B$ par rapport à $p_1$ :
\begin{equation}
  \partial_{p_1} A \circ \partial_{p_1} B = p_1.
\end{equation}

La dualité de Koszul est reliée comme suit à la transformée de Legendre, voir \cite[\S 7]{getzler-kapranov_modular} et en particulier le corollaire (7.22).

\begin{proposition}
  \label{koszul_legendre}
  Lorsque deux opérades cycliques $\oV$ et $\oW$ sont de Koszul et
  duales l'une de l'autre (au sens cyclique) alors la série
  caractéristique $\cha_{\oV}$ est la transformée de Legendre de
  $-\Sigma \cha_{\oW}$.
\end{proposition}

Dans cet énoncé, ``au sens cyclique'' signifie que la structure
cyclique de $\oW$ provient de celle de $\oV$ via la cobar-construction
$\cobar\oV$ vue comme opérade cyclique et le
quasi-isomorphisme entre $\cobar\oV$ et $\oW$. Dans ce cas, on a une égalité
\begin{equation}
  \cha_{\oW} = \cha_{\cobar\oV}.
\end{equation}

Pour plus de détails sur le cadre théorique de la proposition
\ref{koszul_legendre}, le lecteur pourra consulter
\cite{getzler-kapranov_modular,getzler-kapranov_cyclic,getzler_m0n}.

\subsection{Séries caractéristiques de $\oV$ et $\oW$}

\begin{proposition}
  \label{formuleW}
  La série caractéristique de l'opérade cyclique $\oW$ est
  \begin{equation}
    \cha_{\oW} = \sum_{n \geq 1} \frac{1}{2n} \sum_{j | 2n} (-1)^{j(n-1)} \phi(2n/j) p_{2n/j}^j,
  \end{equation}
  où $\phi$ est l'indicatrice d'Euler.
\end{proposition}
\begin{proof}
  Par la définition de $\cha_{\oW}$ et par la proposition \ref{cycliqueW}, on a
  \begin{equation*}
    \cha_{\oW}=\sum_{n \geq 1} \Ind_{\ZZ_{/2n}}^{\sym_{2n}} (-1)^{n-1},
  \end{equation*}
  où $(-1)^{n-1}$ est la caractère valant $(-1)^{n-1}$ sur le
  générateur de $\ZZ_{/2n}$. Par un calcul standard du caractère
  induit, on trouve que
  \begin{equation*}
    \cha_{\oW} = \sum_{n \geq 1} \frac{1}{2n} \sum_{j | 2n} \sum\limits_{\substack{i=1\\i \wedge 2n/j=1}}^{2n/j} (-1)^{ij(n-1)} p_{2n/j}^j.
  \end{equation*}


  Il suffit enfin de remarquer (en distinguant le cas où $j$ est
  impair et $n$ pair) que
  \begin{equation*}
    \sum\limits_{\substack{i=1\\i \wedge 2n/j=1}}^{2n/j} (-1)^{ij(n-1)} = (-1)^{j(n-1)} \phi(2n/j).
  \end{equation*}
\end{proof}

Par la remarque \ref{crux}, la proposition \ref{koszul_legendre} et la
proposition \ref{transfo}, on a donc
\begin{proposition}
  \label{formuleV}
  La série caractéristique de l'opérade cyclique $\oV$ est
  \begin{equation}
    \cha_{\oV} = \sum_{n \geq 1} (-1)^{n-1} c_{n-1} p_1^{2n} + \sum_{n \geq 1} \frac{1}{2n} \sum_{ j| 2n} \lambda({2n/j}) \phi(j) (-1)^{2n(n-1)/j} p_j^{2n/j},
  \end{equation}
où
\begin{equation}
\label{def_lambda} \lambda(n)=(-1)^{\binom{n}{2}} \binom{n-1}{\lfloor \frac{n-1}{2} \rfloor} \quad \text{ et }\quad  c_n=\frac{1}{n+1}\binom{2n}{n}.
\end{equation}
\end{proposition}

\subsection{Polynôme caractéristique de $\thetav$}

On introduit la suite $(b_n)_{n \geq 1}$ définie pour tout $n\geq 1$ par
\begin{equation}
  \label{defib}
  b_n =\frac{1}{n} \sum_{d |n} \mu(d) \lambda(n/d),
\end{equation}
où $\mu$ est la fonction de Möbius et $\lambda$ est définie dans \eqref{def_lambda}. On peut montrer que les $b_n$ sont des entiers relatifs, voir \cite[\S 3]{chapoton_AIF}. Par inversion de Möbius, on a
\begin{equation}
  \label{sommeb}
  \sum_{d |n} d b_d = \lambda(n).
\end{equation}

Soit $n$ un entier et $d$ divisant $n$. On note $M_{n,d}$ le module
$\QQ[t]/(t^d-1)$ sur lequel le générateur de $\ZZ_{/n}$ agit par
multiplication par $t$. Soit $M'_{n,d}$ le module induit de $\ZZ_{/n}$
à $\sym_n$ de $M_{n,d}$. Le caractère de $M'_{n,d}$ est donné (voir \cite[\S 1.4]{chapoton_AIF}) par la
formule :
\begin{equation}
  \label{formuleM}
  \frac{d}{n}\sum_{\ell | n/d} \phi(\ell) p_{\ell}^{n/\ell}.
\end{equation}
En particulier, le caractère de $M'_{n,n}$ est $p_1^n$.

La proposition \ref{formuleV} se traduit en l'assertion suivante.
\begin{proposition}
  \label{somme_de_M}
  L'action du groupe cyclique $\ZZ_{/2n+2}$ engendré par $(-1)^{n+1} \thetav$ sur l'espace vectoriel $\oV(2n+1)$ est donnée par le module virtuel
  \begin{equation}
    c_n M_{2n+2,2n+2} - (-1)^{n+1} \sum_{d | 2n+2} b_d M_{2n+2,d}.
  \end{equation}
\end{proposition}
\begin{proof}
  La formule pour $\cha_{\oV}$ de la prop. \ref{formuleV} peut s'écrire, en utilisant \eqref{sommeb}, comme
  \begin{equation}
    \sum_{n \geq 1} (-1)^{n-1} c_{n-1} p_1^{2n} + \sum_{n \geq 1} \frac{1}{2n} \sum_{ j| 2n} j b_j \sum_{\ell|2n/j} \phi(\ell) (-1)^{2n(n-1)/\ell} p_\ell^{2n/\ell}.
  \end{equation}
  En tenant compte du fait que $\oV(2n-1)$ est en poids $n-1$, on obtient pour l'action de $\gammav=-\thetav$, la formule
  \begin{equation}
    \sum_{n \geq 1} c_{n-1} p_1^{2n} + (-1)^{n-1} \sum_{n \geq 1}  \sum_{ j| 2n} b_j \frac{j}{2n}\sum_{\ell|2n/j} \phi(\ell) (-1)^{2n(n-1)/\ell} p_\ell^{2n/\ell}.
  \end{equation}
   En faisant agir $\omega^{n-1}$ sur la partie de degré $2n$, on obtient
  \begin{equation}
    \sum_{n \geq 1} c_{n-1} p_1^{2n} - (-1)^{n} \sum_{n \geq 1}  \sum_{ j| 2n} b_j \frac{j}{2n}\sum_{\ell|2n/j} \phi(\ell) p_\ell^{2n/\ell},
  \end{equation}
  qui décrit les actions de $(-1)^{n} \thetav$ sur $\oV(2n-1)$, par le lemme
  \ref{involution}.

  Enfin, on reconnaît la formule \eqref{formuleM}, et on obtient l'expression virtuelle
  \begin{equation}
    c_n M'_{2n+2,2n+2} - (-1)^{n+1} \sum_{d | 2n+2} b_d M'_{2n+2,d}
  \end{equation}
  pour la représentation induite de $(-1)^{n+1} \thetav$ sur
  $\oV(2n+1)$. On en déduit l'énoncé en utilisant l'injectivité de
  l'application linéaire $\Ind$ (voir \cite[\S 1.4]{chapoton_AIF}).
\end{proof}

\begin{remark}
  La forme de cette proposition fait penser à l'existence possible
  d'une suite exacte courte. Il serait intéressant de décrire
  explicitement une telle suite exacte.
\end{remark}

\begin{lemma}
  \label{involution}
  On a un carré commutatif
  \begin{equation}
     \xymatrix{
    \ZZ_{/2n}-\mod \ar[r]^\Ind \ar[d]^{\alpha} & \sym_{2n}-\mod \ar[d]^{\omega} \\
    \ZZ_{/2n}-\mod \ar[r]^\Ind & \sym_{2n}-\mod
  }
  \end{equation}
  où $\alpha$ est l'application qui multiplie par $-1$ l'action du
  générateur de $\ZZ_{/2n}$ et $\omega$ est l'involution des fonctions
  symétriques définie par \eqref{def_omega}.
\end{lemma}
\begin{proof}
  La formule pour l'induite d'un caractère $\chi$ est donnée par
  \begin{equation*}
    \frac{1}{2n}\sum_{j | 2n} \sum\limits_{\substack{i=1\\i \wedge 2n/j=1}}^{2n/j} \chi(j i)  p_{2n/j}^j.
  \end{equation*}
  Pour le caractère $\alpha(\chi)$ tordu par le signe, on obtient donc
  \begin{equation*}
    \frac{1}{2n}\sum_{j | 2n} \sum\limits_{\substack{i=1\\i \wedge 2n/j=1}}^{2n/j} (-1)^{j i}\chi(j i)  p_{2n/j}^j.
  \end{equation*}
  Mais dans cette somme, $(-1)^{j i}$ vaut toujours $(-1)^j$, que $j$
  soit pair ou impair. On obtient ainsi la formule pour l'image par
  $\omega$ de l'induite de $\chi$, et donc le carré commutatif voulu.
\end{proof}

On déduit de la proposition \ref{somme_de_M} l'énoncé suivant. 
\begin{theorem}
  Le polynôme caractéristique de la transformation de Coxeter du poset
  de Tamari $\Y_n$ est donné par
  \begin{equation}
    \frac{(x^{2n+2}-1)^{c_n}}{\left(\prod_{d | 2n+2} (x^d - (-1)^{d(n+1)})^{b_d}\right)^{(-1)^{n+1}}}.
  \end{equation}
\end{theorem}
\begin{proof}
  Comme le polynôme caractéristique de $M_{n,d}$ est $x^d-1$, on
  déduit de la proposition \ref{somme_de_M} que le polynôme
  caractéristique de $(-1)^{n+1} \thetav$ sur $\oV(2n+1)$ est
   \begin{equation}
    \frac{(x^{2n+2}-1)^{c_n}}{\left(\prod_{d | 2n+2} (x^d - 1)^{b_d}\right)^{(-1)^{n+1}}}.
  \end{equation}
  Pour obtenir le polynôme caractéristique de $\thetav$ (qui est aussi
  celui de $\theta$ par le théorème \ref{theo_idem}), il suffit de
  remplacer $x$ par $(-1)^{n+1} x$ et de simplifier les signes.
\end{proof}

Ceci démontre les conjectures 3.4 et 3.5 de \cite{chapoton_AIF},
modulo une reformulation simple.

\subsection{Calcul de la transformée de Legendre}

Cette section est consacrée au calcul de la transformée de Legendre de
la fonction symétrique associée à l'opérade cyclique $\oW$.

On introduit la fonction symétrique $A$ (qui correspond à $\cha_\oV$) :
\begin{equation}
  A = \sum_{n \geq 1} (-1)^{n-1} c_{n-1} p_1^{2n} + \sum_{n \geq 1} \frac{1}{2n} \sum_{ j| 2n} \lambda({2n/j}) \phi(j) (-1)^{2n(n-1)/j} p_j^{2n/j},
\end{equation}
où $\lambda$ et $c_n$ sont définis dans \eqref{def_lambda}. On considère aussi la fonction symétrique $B$ (qui correspond à $-\Sigma \cha_\oW$) :
\begin{equation}
  B= \sum_{n \geq 1} \frac{1}{2n} \sum_{ j| 2n} (-1)^{2n n/j} \phi(j) p_{j}^{2n/j}.
\end{equation}

On calcule sans difficulté les dérivées partielles par rapport à $p_1$ :
\begin{equation}
   \partial_{p_1} A = \sum_{n\geq 1} (-1)^{n-1} c_{n-1} p_1^{2n-1}
\end{equation}
et
\begin{equation}
    \partial_{p_1} B = \frac{p_1}{1-p_1^2}.
\end{equation}

On remarque que le premier terme de $A$ est $p_1 \partial_{p_1} A$.

\begin{proposition}
  \label{transfo}
  La fonction $A$ est la transformée de Legendre de $B$.
\end{proposition}
\begin{proof}
  Calculons d'abord $p_1 \partial_{p_1} B - B$ :
  \begin{equation*}
    \frac{p_1^2}{1-p_1^2}-\sum_{n \geq 1} \frac{1}{2n} \sum_{ j| 2n} (-1)^{2n n/j} \phi(j) p_{j}^{2n/j}.
  \end{equation*}
  D'autre part, $A \circ \partial_{p_1} B$ vaut
  \begin{equation*}
    (\partial_{p_1} B) (\partial_{p_1} A \circ \partial_{p_1} B) + \sum_{n \geq 1} \frac{1}{2n} \sum_{ j| 2n} \lambda({2n/j}) \phi(j) (-1)^{2n(n-1)/j} \left(\frac{p_j}{1-p_j^2}\right)^{2n/j},
  \end{equation*}
  ce qui donne
  \begin{equation*}
    \frac{p_1^2}{1-p_1^2} + \sum_{n \geq 1} \frac{1}{2n} \sum_{ j| 2n} \lambda({2n/j}) \phi(j) (-1)^{2n(n-1)/j} \left(\frac{p_j}{1-p_j^2}\right)^{2n/j}.
  \end{equation*}

  Il suffit donc de montrer l'égalité entre les seconds termes. En échangeant les sommations, on obtient d'une part
  \begin{equation*}
    -\sum_{j \geq 1} \frac{\phi(j)}{j} \sum\limits_{\substack{n \geq 1\\j| 2n}} \frac{j}{2n}(-1)^{2n n/j} p_{j}^{2n/j},
  \end{equation*}
  et d'autre part
  \begin{equation*}
    \sum_{j \geq 1} \frac{\phi(j)}{j}\sum\limits_{\substack{n \geq 1\\j| 2n}} \frac{j}{2n} \lambda({2n/j})(-1)^{2n(n-1)/j} \left(\frac{p_j}{1-p_j^2}\right)^{2n/j}.
  \end{equation*}
  Leur égalité résulte du lemme \ref{series_formelles}.
\end{proof}

\begin{lemma}
  \label{series_formelles}
  Pour tout $j\geq 1$ fixé, on a
  \begin{equation}
    -\sum\limits_{\substack{n \geq 1\\j| 2n}} \frac{j}{2n} (-1)^{2n n/j} x^{2n/j} =
    \sum\limits_{\substack{n \geq 1\\j| 2n}} \frac{j}{2n} \lambda({2n/j})(-1)^{2n(n-1)/j} \left(\frac{x}{1-x^2}\right)^{2n/j}.
  \end{equation}
\end{lemma}
\begin{proof}
  On peut réécrire cette égalité comme suit :
  \begin{equation*}
    - \sum\limits_{\substack{N \geq 1\\j N \text{pair}}} \frac{1}{N} (-1)^{j N^2/2} x^N
    = \sum\limits_{\substack{N \geq 1\\j N \text{pair}}} \frac{1}{N} \lambda(N) (-1)^{j N/2+N}\left(\frac{x}{1-x^2}\right)^{N}.
  \end{equation*}
  Pour vérifier cette relation, on distingue le cas $j$ pair du cas $j$ impair.

  Si $j$ est pair, on doit vérifier que
  \begin{equation*}
     - \sum_{N \geq 1} \frac{1}{N} (-1)^{j N/2} x^N
    = \sum_{N \geq 1} \frac{1}{N} \lambda(N) (-1)^{j N/2+N} \left(\frac{x}{1-x^2}\right)^{N}.   
  \end{equation*}
  Ceci résulte aisément du développement de Taylor 
  \begin{equation*}
    \sum_{N \geq 1} \frac{1}{N} \lambda(N) y^{N} = -\log\left( \frac{1-2y+\sqrt{1+4y^2}}{2} \right).
  \end{equation*}

  Si $j$ est impair, on doit vérifier que
  \begin{equation*}
     - \sum_{N \geq 1} \frac{1}{2N} x^{2N}
    = \sum_{N \geq 1} \frac{1}{2N} \lambda(2N) \left(\frac{x}{1-x^2}\right)^{2N}.   
  \end{equation*}
  Ceci résulte du développement de Taylor
  \begin{equation*}
    \sum_{N \geq 1} \frac{1}{2N} \binom{2N}{N} y^{N} = -\log\left( \frac{1+\sqrt{1-4y^2}}{2} \right).
  \end{equation*}

  Pour ces deux développements de Taylor, le lecteur peut consulter
  \cite[Append. A]{chapoton_AIF}.
\end{proof}

\bibliographystyle{plain}
\bibliography{ternaire}

\begin{thebibliography}{10}

\bibitem{aguiar-sottile}
Marcelo Aguiar and Frank Sottile.
\newblock Structure of the {L}oday-{R}onco {H}opf algebra of trees.
\newblock {\em J. Algebra}, 295(2):473--511, 2006.

\bibitem{coxeter_tamari}
Fr{\'e}d{\'e}ric Chapoton.
\newblock On the {C}oxeter transformations for {T}amari posets.
\newblock {\em Canad. Math. Bull.}, 50(2):182--190, 2007.

\bibitem{chapoton_AIF}
Fr{\'e}d{\'e}ric Chapoton.
\newblock Le module dendriforme sur le groupe cyclique.
\newblock {\em Ann. Inst. Fourier (Grenoble)}, 58(7):2333--2350, 2008.

\bibitem{categorifi}
Fr{\'e}d{\'e}ric Chapoton.
\newblock Categorification of the dendriform operad.
\newblock In {J}ean-{L}ouis {L}oday and {B}runo {V}allette, editors, {\em
  Proceedings of Operads 2009}, S{\'e}minaire et Congr{\`e}s. SMF, 2012.
\newblock oai:arXiv.org:0909.2751.

\bibitem{curtis-reiner}
Charles~W. Curtis and Irving Reiner.
\newblock {\em Representation theory of finite groups and associative
  algebras}.
\newblock Pure and Applied Mathematics, Vol. XI. Interscience Publishers, a
  division of John Wiley \& Sons, New York-London, 1962.

\bibitem{dotsenko}
Vladimir Dotsenko and Anton Khoroshkin.
\newblock Gr\"obner bases for operads.
\newblock {\em Duke Math. J.}, 153(2):363--396, 2010.

\bibitem{k1}
Kurusch Ebrahimi-Fard and Dominique Manchon.
\newblock Dendriform equations.
\newblock {\em J. Algebra}, 322(11):4053--4079, 2009.

\bibitem{k3}
Kurusch Ebrahimi-Fard, Dominique Manchon, and Fr{\'e}d{\'e}ric Patras.
\newblock New identities in dendriform algebras.
\newblock {\em J. Algebra}, 320(2):708--727, 2008.

\bibitem{fominz}
Sergey Fomin and Andrei Zelevinsky.
\newblock Cluster algebras. {I}. {F}oundations.
\newblock {\em J. Amer. Math. Soc.}, 15(2):497--529 (electronic), 2002.

\bibitem{friedman-tamari}
Haya Friedman and Dov Tamari.
\newblock Probl\`emes d'associativit\'e: {U}ne structure de treillis finis
  induite par une loi demi-associative.
\newblock {\em J. Combinatorial Theory}, 2:215--242, 1967.

\bibitem{getzler_m0n}
E.~Getzler.
\newblock Operads and moduli spaces of genus {$0$} {R}iemann surfaces.
\newblock In {\em The moduli space of curves ({T}exel {I}sland, 1994)}, volume
  129 of {\em Progr. Math.}, pages 199--230. Birkh\"auser Boston, Boston, MA,
  1995.

\bibitem{getzler-kapranov_cyclic}
E.~Getzler and M.~M. Kapranov.
\newblock Cyclic operads and cyclic homology.
\newblock In {\em Geometry, topology, \& physics}, Conf. Proc. Lecture Notes
  Geom. Topology, IV, pages 167--201. Int. Press, Cambridge, MA, 1995.

\bibitem{getzler-kapranov_modular}
E.~Getzler and M.~M. Kapranov.
\newblock Modular operads.
\newblock {\em Compositio Math.}, 110(1):65--126, 1998.

\bibitem{happel}
Dieter Happel.
\newblock {\em Triangulated categories in the representation theory of
  finite-dimensional algebras}, volume 119 of {\em London Mathematical Society
  Lecture Note Series}.
\newblock Cambridge University Press, Cambridge, 1988.

\bibitem{happel-unger}
Dieter Happel and Luise Unger.
\newblock On a partial order of tilting modules.
\newblock {\em Algebr. Represent. Theory}, 8(2):147--156, 2005.

\bibitem{hnt}
F.~Hivert, J.-C. Novelli, and J.-Y. Thibon.
\newblock The algebra of binary search trees.
\newblock {\em Theoret. Comput. Sci.}, 339(1):129--165, 2005.

\bibitem{hoffbeck}
Eric Hoffbeck.
\newblock A {P}oincar\'e-{B}irkhoff-{W}itt criterion for {K}oszul operads.
\newblock {\em Manuscripta Math.}, 131(1-2):87--110, 2010.

\bibitem{huang-tamari}
Samuel Huang and Dov Tamari.
\newblock Problems of associativity: {A} simple proof for the lattice property
  of systems ordered by a semi-associative law.
\newblock {\em J. Combinatorial Theory Ser. A}, 13:7--13, 1972.

\bibitem{lad2}
Sefi Ladkani.
\newblock Universal derived equivalences of posets of cluster tilting objects,
  2007.

\bibitem{lad1}
Sefi Ladkani.
\newblock Universal derived equivalences of posets of tilting modules, 2007.

\bibitem{ladkani_poset}
Sefi Ladkani.
\newblock On derived equivalences of categories of sheaves over finite posets.
\newblock {\em J. Pure Appl. Algebra}, 212(2):435--451, 2008.

\bibitem{lenzing}
Helmut Lenzing.
\newblock Coxeter transformations associated with finite-dimensional algebras.
\newblock In {\em Computational methods for representations of groups and
  algebras (Essen, 1997)}, volume 173 of {\em Progr. Math.}, pages 287--308.
  Birkh\"auser, Basel, 1999.

\bibitem{loday-dialgebras}
Jean-Louis Loday.
\newblock Dialgebras.
\newblock In {\em Dialgebras and related operads}, volume 1763 of {\em Lecture
  Notes in Math.}, pages 7--66. Springer, Berlin, 2001.

\bibitem{arithmetree}
Jean-Louis Loday.
\newblock Arithmetree.
\newblock {\em J. Algebra}, 258(1):275--309, 2002.
\newblock Special issue in celebration of Claudio Procesi's 60th birthday.

\bibitem{loday-ronco}
Jean-Louis Loday and Mar{\'{\i}}a~O. Ronco.
\newblock Order structure on the algebra of permutations and of planar binary
  trees.
\newblock {\em J. Algebraic Combin.}, 15(3):253--270, 2002.

\bibitem{macdonald}
I.~G. Macdonald.
\newblock {\em Symmetric functions and {H}all polynomials}.
\newblock Oxford Mathematical Monographs. The Clarendon Press Oxford University
  Press, New York, second edition, 1995.
\newblock With contributions by A. Zelevinsky, Oxford Science Publications.

\bibitem{markl_remm}
Martin Markl and Elisabeth Remm.
\newblock (non-)koszulity of operads for n-ary algebras, cohomology and
  deformations, 2009.

\bibitem{markl}
Martin Markl, Steve Shnider, and Jim Stasheff.
\newblock {\em Operads in algebra, topology and physics}, volume~96 of {\em
  Mathematical Surveys and Monographs}.
\newblock American Mathematical Society, Providence, RI, 2002.

\bibitem{cambrien}
Nathan Reading.
\newblock Cambrian lattices.
\newblock {\em Adv. Math.}, 205(2):313--353, 2006.

\bibitem{riedtmann}
Christine Riedtmann and Aidan Schofield.
\newblock On a simplicial complex associated with tilting modules.
\newblock {\em Comment. Math. Helv.}, 66(1):70--78, 1991.

\bibitem{ronco}
Mar{\'{\i}}a Ronco.
\newblock Primitive elements in a free dendriform algebra.
\newblock In {\em New trends in {H}opf algebra theory ({L}a {F}alda, 1999)},
  volume 267 of {\em Contemp. Math.}, pages 245--263. Amer. Math. Soc.,
  Providence, RI, 2000.

\bibitem{tamari}
Dov Tamari.
\newblock The algebra of bracketings and their enumeration.
\newblock {\em Nieuw Arch. Wisk. (3)}, 10:131--146, 1962.

\bibitem{lodayv}
Bruno Vallette and Jean-Louis Loday.
\newblock {\em Algebraic Operads}.
\newblock a paraitre, 2010.
\newblock xviii+512 pp.

\end{thebibliography}

\end{document}